\documentclass[11pt,bezier]{article}
\usepackage{amsmath,amssymb,amsfonts,euscript,graphicx,algorithm,algorithmic,makeidx,amscd,enumerate,setspace,tkz-berge, amscd}

\newtheorem{prethm}{{\bf Theorem}}

\newenvironment{thm}{\begin{prethm}{\hspace{-0.5
em}{\bf.}}}{\end{prethm}}

\newtheorem{prepro}[prethm]{{\bf Theorem}}

\newtheorem{preprop}[prethm]{{\bf Proposition}}

\newenvironment{prop}{\begin{preprop}{\hspace{-0.5
em}{\bf.}}}{\end{preprop}}

\newtheorem{precor}[prethm]{{\bf Corollary}}

\newenvironment{cor}{\begin{precor}{\hspace{-0.5
em}{\bf.}}}{\end{precor}}

\newtheorem{predefinition}[prethm]{{\bf Definition}}

\newenvironment{definition}{\begin{predefinition}{\hspace{-0.5
em}{\bf.}}}{\end{predefinition}}

\newtheorem{preconj}[prethm]{{\bf Conjecture}}

\newtheorem{preremark}[prethm]{{\bf Remark}}

\newenvironment{remark}{\begin{preremark}\rm{\hspace{-0.5
em}{\bf.}}}{\end{preremark}}

\newtheorem{preexample}[prethm]{{\bf Example}}

\newenvironment{example}{\begin{preexample}\rm{\hspace{-0.5
em}{\bf.}}}{\end{preexample}}

\newtheorem{prelem}[prethm]{{\bf Lemma}}

\newenvironment{lem}{\begin{prelem}{\hspace{-0.5
em}{\bf.}}}{\end{prelem}}

\newtheorem{prelam}{{\bf Lemma}}

\newtheorem{preproof}{{\bf Proof.}}

\newenvironment{proof}[1]{\begin{preproof}{\rm
#1}\hfill{$\Box$}}{\end{preproof}}

\title{\bf \large Tilting Modules Over Gorensetein $T_n^d$-Injective Gorensetein $T_n^d$-flat Modules
\thanks {2010{ \it Mathematics Subject Classification}: 13D07; 16D40; 18G25;}}

\author{{\normalsize M. Amini${}$  {}}\vspace{3mm}\\
{\footnotesize{${}^{\mathsf{}}$\it }}\\
{\footnotesize{${}^{\mathsf{}}$\it Department of Mathematics, Payame Noor University, Tehran, Iran.}}\\
{\footnotesize{}}\\
{\footnotesize{}}
{\footnotesize{$\mathsf{}$\quad\quad
$\mathsf{ mostafa.amini@pnu.ac.ir}$\quad\quad $\mathsf{}$}}}

\date{}

\begin{document}
\maketitle
{\small\noindent �{\bf{Abstract.}} Let  $T$ be a tilting module. In this paper, { Gorenstein $T_n^{d}$-injective}  and {Gorenstein $T_n^{d}$-flat} modules are introduced. If $G\in { \rm Cogen}T$ (resp; $G\in {\rm  Gen}T$), then $G$  is called  {Gorenstein $T_n^{d}$-injective (resp; Gorenstein $T_n^{d}$-flat) }  if there exists the exact sequence 
${\mathbf{M}}= \cdots\rightarrow M_1\rightarrow M_{0}\rightarrow M^0\rightarrow
M^1\rightarrow\cdots$ of  ${\mathcal{ET}}$-modules (resp; ${\mathcal{FT}}$-modules)  with $G=\ker(M^0\rightarrow M^1)$ such that ${\rm {\mathcal{E}}}_{T}^{d}(U,{\mathbf{M}})$ (resp; ${\Gamma}_{d}^{T}(U,{\mathbf{F}})$) leaves this sequence exact whenever $U\in {\rm  F.Pres}^{n}T$ with ${\rm T.pdim}(U)<\infty$ (resp; ${ \rm T.fdim}(U)<\infty$).  Also, some characterizations of rings over  Gorenstein $T_n^{d}$-injective  and Gorenstein $T_n^{d}$-flat modules are given.

\vspace{5mm} \noindent{\bf\large 1. Basic Definitions and Notations}\\\\
Throughout this paper,
$R$ is an associative ring with non-zero identity, all modules are
unitary left $R$-modules. First we recall some known notions and facts needed in the sequel. Let $R$ be a ring and $T$ an R-module. Then 
\begin{enumerate}
\item[(1)] A module $M$ is said to be \textit{cogenerated},
by $T$, denoted by $M\in {Cogen}T$,
(resp; \textit{generated}, denoted $M\in{Gen}T$) by $T$
if there exists an exact sequence $0\rightarrow M\rightarrow T^n$ (resp; $ T^{(n)}\rightarrow M\rightarrow 0 $ ),
 for some positive integer $n$.
\item[(2)] We denote by
${Prod}T$ (resp;
${F.Prod}T$), the class of modules isomorphic to direct
summands of direct product of copies (resp; finitely many copies) of $T$.

\item[(3)] We denote by
${Add}T$ (resp;
${F.Add}T$), the class of modules isomorphic to direct
summands of direct sum of copies (resp; finitely many copies) of $T$.

\item[(4)] By
$Copres^{n}T$ (resp; $F.Copres^{n}T$) and $Copres^{\infty}T$ (resp; $F.Copres^{\infty}T$), we denote the set of all modules $M$ such
that there exists exact sequences $$\begin{array}{cccccccccccc}
\vspace{-.10cm} 0 \longrightarrow M\longrightarrow
T_{0}\longrightarrow T_{1}\longrightarrow\cdots\longrightarrow T_{n-1}\longrightarrow T_{n}
\end{array}$$
and
$$\begin{array}{cccccccccccc} \vspace{-.10cm}
0 \longrightarrow M\longrightarrow
T_{0}\longrightarrow T_{1}\longrightarrow\cdots\longrightarrow T_{n-1}\longrightarrow T_{n}\longrightarrow\cdots
,\end{array}$$ respectively, where $T_{i}\in {\rm Prod}T$ (resp. $T_{i}\in{\rm F.Prod}T$), for
every $i\geq 0$. 

\item[(5)]  Following \cite{qv}, a module $T$ is called {\it tilting} if it satisfies the following conditions:\\
(a) $ pd(T)\leq 1$, where $pd(T)$ denotes the {\it projective dimension of $T$}.\\
(b)$ Ext^{i}(T,T^{(\lambda)})=0$, for each $i>0$ and for every cardinal $\lambda$.\\
(c) There exists the exact sequence $0\rightarrow R\rightarrow
T_{0}\rightarrow T_{1}\rightarrow 0$, where $T_0,T_1\in { \rm Add}T$.

\item[(6)] 
 For any tilting module $T$, if $N\in {\rm Cogen}T$ and $M\in {\rm Gen}T$ , then \cite[Proposition 2.1]{shaveisicam} implies that ${\rm Cogen}T={\rm Copres}^\infty T$ and ${\rm Gen}T={\rm Pres}^\infty T$ .
This shows that any module cogenerated by $T$ and  any module generated by $T$ has an ${\rm Prod}T$-resolution and ${\rm Add}T$-resolution.

\item[(7)] For any homomorphism $f$, we denote by ${ker}f$ and ${im}f$, the kernel and image of $f$, respectively. Let  $N\in {\rm
Cogen}T$ and $M\in {\rm
Gen}T$ be two modules, where $T$ is tilting module. We define the functors $$ {\Gamma}_{n}^{T}(M,-):= \frac{{\rm ker}(
\delta_{n}\otimes 1_{B})}{{\rm im}(\delta_{n+1}\otimes 1_{B})};\ \ {\mathcal{E}}_{T}^{n}(-,N):= \frac{{\rm ker}\delta_{*}^{n}}{{\rm
im} \delta_{*}^{n-1}},$$ where
$$\begin{array}{cccccccccc} \vspace{-.3cm}
&&&\delta_{2}&&\delta_{1}&&\delta_{0}\\
\cdots& \longrightarrow&T_{2}&\longrightarrow&T_{1}& \longrightarrow & T_{0}
&\longrightarrow & M &\longrightarrow 0
\end{array}$$ is an
${\rm Add}T$-resolution of $M$, 

and
$$\begin{array}{cccccccccc} \vspace{-.3cm}
&\delta_{0}&&\delta_{1}&&\delta_{2}\\
0\longrightarrow N& \longrightarrow&T^{0}&\longrightarrow&T^{1}& \longrightarrow &\cdots
\end{array}$$ is an
${\rm Prod}T$-resolution of $N$ and $\delta_{*} ^{n}={ Hom}(\delta_n,id_B)$,
for every
$i\geq 0$.

 Let $M\in {\rm Gen}T$ be a module. A similar proof to that of \cite[Lemma 2.11]{seam1} shows that ${\mathcal{E}}^0_T(M,-)\cong {\rm Hom}(M,-)$. Similarly, it is seen that $\Gamma_T^0(M,-)\cong M\otimes -$. Moreover, ${\mathcal{E}}^1_T(M,-)=0$ implies that $M\in {\rm Add}T$. If $N\in {\rm Cogen}T$, then ${\mathcal{E}}^1_T(-,N)=0$ implies that $N\in {\rm Prod}T$.
It is clear that ${\rm T.pdim}(M)=n$ if and only if $n$ is the least non-negative integer such that
${\mathcal{E}}^{n+1}_{T}(M,B)=0$, for
any module $B$. 
Naturally, we say that $M$ has
{\it $T$-flat dimension} ({\it $T$-injective dimension}) $n$, denoted by ${ T.fdim}(M)=n$ (${ T.idim}(N)=n$) if $n$ is the least non-negative integer such that
${\Gamma}_{n+1}^{T}(M,B)=0$ (${\mathcal{E}}^{n+1}_{T}(B,N)=0$), for
any module $B$. A module with zero $T$-projective (resp., $T$-injective) dimension is called
{\it $T$-projective {\rm(}resp., $T$-injective{\rm)}}, see \cite{shaveisicam, seam3}.

\item[(8)]
$M$ is said to be $n$-{\it presented} \cite{ N.I, X.J, wu} if there is an exact sequence of $R$-modules
 $ F_{n}\rightarrow F_{n-1}\rightarrow\dots\rightarrow F_1\rightarrow F_0\rightarrow
M\rightarrow$ , where each $F_i$ is a finitely generated free, equivalently projective, $R$-module.
\item[(9)]
$R$ is said to be  $n$-{\it coherent} \cite{N.I, X.J, wu} if every $n$-presented $R$-module
is $(n + 1)$-presented.
\item[(10)]
 $M$  is said to be {\it Gorenstein flat} (resp.,{\it Gorenstein injective})
\cite{ E.J, E.B} if there
is an exact sequence $\cdots\rightarrow I_1\rightarrow I_{0}\rightarrow I^0\rightarrow
I^1\rightarrow\cdots$ of flat (resp., injective ) modules
with $M= {\rm ker}(I^0\rightarrow
I^1)$ such that $U\otimes_{R}-$ (resp; ${\rm Hom}(U,-)$) leaves the sequence exact whenever
$U$ is an injective  module.
\end{enumerate}

\ \ \ A module $M$ is called  {\it $T_n^{d}$-injective} if ${\mathcal{E}}_{T}^{d+1}(U,M)=0$ for every $U\in {\rm F.Pres}^{n}T$. A module $M$ is called  {\it $T_n^{d}$-flat} if ${\Gamma}_{d+1}^{T}(U,M)=0$ for every $U\in {\rm F.Pres}^{n}T$ see \cite{ seam2}.  We denote by ${\mathcal{ET}}$ and ${\mathcal{FT}}$  the class of  $T_n^{d}$-injective modules  belong to ${\rm Cogen}T$ and $T_n^{d}$-flat modules belong to ${\rm Gen}T$, respectivily.
In this paper, $T$ is a tilting module. We introduce the {\it Gorenstein $T_n^{d}$-injective} and {\it Gorenstein $T_n^{d}$-flat} 
 modules. A module $G\in {\rm Cogen}T$  is called  Gorenstein $T_n^{d}$-injective  if there exists the following  exact sequence of  ${\mathcal{ET}}$-modules:
$${\mathbf{M}}= \cdots\longrightarrow M_1\longrightarrow M_{0}\longrightarrow M^0\longrightarrow
M^1\longrightarrow\cdots$$ with $G=\ker(M^0\rightarrow M^1)$ such that ${\mathcal{E}}_{T}^{d}(U,{\mathbf{M}})$ leaves this sequence exact whenever $U\in {\rm F.Pres}^{n}T$ with ${\rm T.pdim}(U)<\infty.$
A module $G\in{\rm Gen}T$ is said to be Gorenstein $T_n^{d}$-flat
 if there exists an exact sequence of ${\mathcal{FT}}$-modules:
$${\mathbf{N}}= \cdots\longrightarrow N_1\longrightarrow N_{0}\longrightarrow N^0\longrightarrow
N^1\longrightarrow\cdots$$  with $G=\ker(N^0\rightarrow N^1)$ such that ${\Gamma}_{d}^{T}(U,{\mathbf{N}})$ leaves this sequence exact whenever $U\in {\rm F.Pres}^{n}T$ with ${\rm T.fdim}(U)<\infty$.
Replacing $T$ by $R$ as an $R$-module, every Gorenstein $T_0^{0}$-injective $R$-module is Gorenstein injective, and every Gorenstein $T_1^{0}$-flat is Gorenstein flat.

A ring $R$  is called  $(n,T)$-{\it coherent} if ${\rm F.Pres}^{n}T={\rm F.Pres}^{n+1}T$.
In Section 2, we study some basic properties of the Gorenstein $T_n^{d}$-flat and Gorenstein $T_n^{d}$-injective modules.
Then some characterizations of $(n,T)$-coherent
rings over Gorenstein $T_n^{d}$-injective and Gorenstein $T_n^{d}$-flat modules are given.

\vspace{5mm}
{\bf\large \noindent 2.  Main Results}

We start with the following lemma.
\begin{lem}\label{2.io}
Let $ 0\rightarrow A\stackrel{\displaystyle f}\rightarrow B\stackrel{\displaystyle g}\rightarrow C\rightarrow 0$ be an exact sequence. Then
\begin{enumerate}
\item [\rm (1)]
If $A\in{\rm Pres}^{n}T$ and $C\in{\rm Pres}^{n}T$, then $B\in{\rm Pres}^{n}T.$
\item [\rm (2)]
If $A\in{\rm Pres}^nT$ and $B\in{\rm Pres}^{n+1}T$, then $C\in{\rm Pres}^{n+1}T.$
\item [\rm (3)]
If $B\in{\rm Pres}^nT$ and $C\in{\rm Pres}^{n+1}T$, then $A\in{\rm Pres}^nT.$
\end{enumerate}
\end{lem}
\begin{proof}
{(1)
We prove the assertion by induction on $n$. If $n=0$, then the commutative
diagram with exact rows
$$
 \begin{array}{ccccccccc}
    0\longrightarrow T^{'}_0& \stackrel{\displaystyle i_0}\longrightarrow T^{'}_0\oplus T^{''}_0&\stackrel{\displaystyle \pi_0}\longrightarrow T^{''}_0\longrightarrow 0 \\
      \hspace{1.4cm} \downarrow h_{0}^{'}&\ \ \ \ \ \ \ \ \ \ \downarrow h_{0}&\ \  \downarrow  h_{0}^{''} &  \\
    0\longrightarrow  A &{\hspace{-.6cm}\stackrel{\displaystyle f} \longrightarrow} \ \ \ \ \  B&{\hspace{-.3cm}\stackrel{\displaystyle g}\longrightarrow} \ \ C\longrightarrow 0\\
 \hspace{1.0cm}  \downarrow  & \ \ \ \ \ \downarrow & {\hspace{-4mm}\downarrow} &  \\
\hspace{1.0cm}  0 &\ \ \ \ \ 0 &   {\hspace{-4mm}0}&
     \end{array}
$$

\noindent exists, where $T'_0,T''_0\in {\rm Add}T$, $i_0$ is the inclusion map, $\pi_0$ is a canonical epimorphism and $h_0=fh_0'$ is epimorphism, by Five Lemma. Let $K_1'=\ker h_0'$, $K_1=\ker h_0$ and $K_1''=\ker h_0''$. So,  $B\in {\rm Pres}^{0}T$. It is clear that $K_1',K_1''\in{\rm Pres}^nT$; so, the induction hypotises implies that $K_1\in {\rm Pres}^nT$. Hence $B\in {\rm Pres}^{n}T$.

(2) First assume that $n=0$. If $B\in{\rm Pres}^1T$ and $A\in{\rm Pres}^0T$,  then the following commutative diagram with exact rows:
\begin{center}
$
 \begin{array}{ccccccccc}
  &\ \ \ \ \ \ T^{'}_0 &\longrightarrow A\longrightarrow 0 \\
     &\ \ \ \ \ \ \ \downarrow\gamma &\downarrow f &  \\
     {\hspace{4cm}T_1}& \stackrel{\displaystyle \alpha_2}\longrightarrow T_{0}&\stackrel{\displaystyle \alpha_1}\longrightarrow B\longrightarrow 0 \\
      & \ \ \ \ \ \|&\downarrow g &  \\
      {\hspace{4cm}T^{'}_{0}\oplus T_1}&\stackrel{\displaystyle h}\longrightarrow T_{0}&\stackrel{\displaystyle g\alpha_1}\longrightarrow C\longrightarrow 0\\
 &&\hspace{-2mm}\downarrow &  \\
  & &  \hspace{-2mm}  0&
     \end{array}
$
\end{center}
\noindent in which the existence of $\gamma$ follows from the exactness of the sequence $${\rm Hom}(T^{'}_0,T_{0}) \rightarrow {\rm Hom}(T^{'}_0,B)\rightarrow 0,$$ since $T^{'}_0$ is $T$-projective.  Also, $h$ is defined by 
$h(t^{'}_0,t_1)= \gamma(t^{'}_0) + \alpha(t_1)$. Therefore, we deduce that $C\in{\rm Pres}^1T$.
For $n>0$, the assertion follows from induction.

(3) This is proved similarly.
%
}
\end{proof}
\begin{definition}\label{2.76}
Let $G$ be a module.
\begin{enumerate}
\item [\rm (1)]
If $G\in {\rm Cogen}T$, then $G$  is called  Gorenstein $T_n^{d}$-injective  if there exists the following  exact sequence of  ${\mathcal{ET}}$-modules:
$${\mathbf{M}}= \cdots\longrightarrow M_1\longrightarrow M_{0}\longrightarrow M^0\longrightarrow
M^1\longrightarrow\cdots$$ with $G=\ker(M^0\rightarrow M^1)$ such that ${\mathcal{E}}_{T}^{d}(U,{\mathbf{M}})$ leaves this sequence exact whenever $U\in {\rm F.Pres}^{n}T$ with ${\rm T.pdim}(U)<\infty.$
\item [\rm (2)]
If $G\in {\rm Gen}T$, then $G$  is called  Gorenstein $T_n^{d}$-flat  if there exists the following exact sequence of ${\mathcal{FT}}$-modules:
$${\mathbf{N}}= \cdots\longrightarrow N_1\longrightarrow N_{0}\longrightarrow N^0\longrightarrow
N^1\longrightarrow\cdots$$ with $G=\ker(N^0\rightarrow N^1)$ such that ${\Gamma}_{d}^{T}(U,{\mathbf{N}})$ leaves this sequence exact whenever $U\in {\rm F.Pres}^{n}T$ with ${\rm T.fdim}(U)<\infty$.
\end{enumerate}
\end{definition}

In the following theorem, we show that in the case of $(n,T)$-coherent rings, the existence of ${\mathcal{FT}}$-complex and  ${\mathcal{ET}}$-complex of a module is sufficient to be Gorenstein $T_n^{d}$-flat and Gorenstein $T_n^{d}$-injective.

\begin{thm}\label{2.3}
Let $R$ be a $(n,T)$-coherent. Then
\begin{enumerate}
\item [\rm (1)]
 $G\in{\rm Cogen}T$ is
Gorenstein $T_n^{d}$-injective if and only if there is an exact sequence
$${\mathbf{M}}= \cdots\longrightarrow M_1\longrightarrow M_{0}\longrightarrow M^0\longrightarrow
M^1\longrightarrow\cdots$$
of ${\mathcal{ET}}$-modules such that $G=\ker(M^0\rightarrow M^1)$.
\item [\rm (2)]
 $G\in{\rm Gen}T$ is Gorenstein $T_n^{d}$-flat if and only if there is an exact sequence 
$${\mathbf{N}}= \cdots\longrightarrow N_1\longrightarrow N_{0}\longrightarrow N^0\longrightarrow
N^1\longrightarrow\cdots$$ of  ${\mathcal{FT}}$-modules  such that $G=\ker(N^0\rightarrow N^1)$.

\end{enumerate}
\end{thm}
\begin{proof}
{(1) ($\Longrightarrow$) : This is a direct consequence of definition.

($\Longleftarrow$) :  By definition, it suffices to show that ${\mathcal{E}}_{T}^{d}(U,{\mathbf{M}})$ is exact for
every module $U\in {\rm F.Pres}^{n}T$  with ${\rm T.pdim} (U)< \infty$. To prove this, we use the induction on $d$. Let $d=0$ and ${\rm T.pdim} (U) = m$, then we  show that ${\rm Hom}(U,{\mathbf{M}})$ is exact.
 To prove this, we use the induction on $m$. The case $m =0$ is clear. Assume that $m \geq 1$. Since ${\rm T.pdim} (U) = m$, there exists an exact sequence
$0\rightarrow L\rightarrow T_{0}\rightarrow U \rightarrow 0$ with $T_{0}\in {\rm F.Add}T\subseteq{\rm F.Pres} ^{n-1}T$. Now, from the $(n,T)$-coherence of $R$ and Lemma \ref{2.io}, we deduce that $L, T_{0}\in {\rm F.Pres} ^{n}T$. Also, ${\rm T.pdim}(L) \leq m-1$ and ${\rm T.pdim} (T_0) =0$. So, the following short exact sequence of complexes exists:
\begin{center}
$
\begin{array}{ccccccccc}
& \vdots&\vdots &\vdots&\\
& \downarrow & \downarrow &\downarrow & \\
0 \longrightarrow &{\rm Hom}(U,M_1) & \longrightarrow {\rm Hom}(T_0,M_1)&\longrightarrow {\rm Hom}(L,M_1)\longrightarrow 0 \\
& \downarrow & \downarrow &\downarrow & \\
0 \longrightarrow &{\rm Hom}(U,M_0)& \longrightarrow {\rm Hom}(T_0,M_0)&\longrightarrow {\rm Hom}(L,M_0)\longrightarrow 0 \\
& \downarrow &\downarrow &\downarrow & \\
0 \longrightarrow &{\rm Hom}(U,M^0)& \longrightarrow {\rm Hom}(T_0,M^0)&\longrightarrow {\rm Hom}(L,M^0)\longrightarrow 0 \\
& \downarrow &\downarrow &\downarrow & \\
0 \longrightarrow &{\rm Hom}(U,M^1)& \longrightarrow {\rm Hom}(T_0,M^1)&\longrightarrow {\rm Hom}(L,M^1)\longrightarrow 0 \\
& \downarrow &\downarrow &\downarrow & \\
& \vdots&\vdots &\vdots&\\
& \parallel & \parallel &\parallel& \\
0 \longrightarrow &{\rm Hom}(U,{\mathbf{M}})& \longrightarrow {\rm Hom}(T_0,{\mathbf{M}})&\longrightarrow {\rm Hom}(L,{\mathbf{M}})\longrightarrow 0. \\

\end{array}
$
\end{center}
\noindent By induction, $ {\rm Hom}(L,{\mathbf{M}})$ and ${\rm Hom}(T_0,{\mathbf{M}})$ are exact, hence ${\rm Hom}(U,{\mathbf{M}})$ is exact by \cite[Theorem 6.10]{rotman}. 

Let $d \geq 1$ and $U\in {\rm F.Pres}^{n}T$. Consider the exact sequence $0\rightarrow K\rightarrow T_{0}\rightarrow U \rightarrow 0$, where $T_{0}\in {\rm F.Add}T$. So the following short exact sequence of complexes exists:
\begin{center}
$
\begin{array}{ccccccccc}
\vdots &\vdots&\\
\downarrow &\downarrow & \\
0 \longrightarrow  {\mathcal{E}}_{T}^{d-1}(K,M_1)&\longrightarrow {\mathcal{E}}_{T}^{d}(U,M_1)\longrightarrow 0 \\
 \downarrow &\downarrow & \\
0 \longrightarrow  {\mathcal{E}}_{T}^{d-1}(K,M_0)&\longrightarrow {\mathcal{E}}_{T}^{d}(U,M_0)\longrightarrow 0 \\\downarrow &\downarrow & \\
0 \longrightarrow {\mathcal{E}}_{T}^{d-1}(K,M^0)&\longrightarrow {\mathcal{E}}_{T}^{d}(U,M^0)\longrightarrow 0 \\
\downarrow &\downarrow & \\
0 \longrightarrow {\mathcal{E}}_{T}^{d-1}(K,M^1)&\longrightarrow {\mathcal{E}}_{T}^{d}(U,M^1)\longrightarrow 0 \\
\downarrow &\downarrow & \\
\vdots &\vdots&\\
 \parallel &\parallel& \\
0 \longrightarrow {\mathcal{E}}_{T}^{d-1}(K,{\mathbf{M}})&\longrightarrow {\mathcal{E}}_{T}^{d}(U,{\mathbf{M}})\longrightarrow 0. \\

\end{array}
$
\end{center}

\noindent By induction, ${\mathcal{E}}_{T}^{d-1}(K,{\mathbf{M}})$ is exact. So, ${\mathcal{E}}_{T}^{d}(U,{\mathbf{M}})$ is exact and hence, $G$ is Gorenstein $T_n^{d}$-flat.

(2) A similar proof to that of (1).
}

\end{proof}

\begin{remark}\label{2}
\begin{enumerate}
\item [\rm (1)]
If $U\in {\rm F.Pres}^{n}T$, then $U\in {\rm F.Pres}^{m}T$ for any $n\geq m.$
\item [\rm (2)]
Every $T_m^{d}$-injective $R$-module is $T_n^{d}$-injective, for any $n\geq m.$
\item [\rm (3)]
Direct sum of $T_n^{d}$-injective $R$-modules is $T_n^{d}$-injective.
\item [\rm (4)]
Every $T_m^{d}$-flat $R$-module is $T_n^{d}$-flat, for any $n\geq m.$
\end{enumerate}
\end{remark}

\begin{cor}\label{2.h}
Let $R$ be an $(n,T)$-coherent ring and $G\in{\rm Cogen}T$  a module. Then the
following assertions are equivalent:
\begin{enumerate}
\item [\rm (1)]
$G$ is Gorenstein $T_n^{d}$-injective;
\item [\rm (2)]
There is an exact sequence $\cdots\rightarrow M_1\rightarrow M_{0}\rightarrow G\rightarrow 0$ of modules, where every $M_i\in{\mathcal{ET}}$;
\item [\rm (3)]
There is a short exact sequence $0\rightarrow L\rightarrow N\rightarrow G\rightarrow 0$ of modules, where
$N\in{\mathcal{ET}}$ and $L $ is Gorenstein $T_n^{d}$-injective.
\end{enumerate}
\end{cor}
\begin{proof}
{
$(1)\Longrightarrow (2)$ and $(1)\Longrightarrow (3)$ follow from definition.

$(2)\Longrightarrow (1)$ For any module $G\in{\rm Cogen}T$, there is an exact sequence
$$0\longrightarrow G\longrightarrow T^{0}\longrightarrow T^1\longrightarrow \cdots$$
where any $T^{i}\in{\rm Prod}T\subseteq{\mathcal{ET}}$.  So, the exact sequence
$$ \cdots\longrightarrow M_1\longrightarrow M_{0}\longrightarrow T^0\longrightarrow
T^1\longrightarrow\cdots$$
of ${\mathcal{ET}}$- modules  exists, where $G={\rm ker}(T^0\rightarrow T^1)$. Therefore, $G$ is Gorenstein $T_n^{d}$-injective, by Theorem \ref{2.3}.

$(3)\Longrightarrow (2)$ Assume that the exact sequence
$$ 0\longrightarrow L\longrightarrow N \longrightarrow G\longrightarrow 0 \ \ (1)$$ exists, where
$N\in{\mathcal{ET}}$ and $L $ is Gorenstein $T_n^{d}$-injective. Since $L$ is Gorenstein $T_n^{d}$-injective, there is an exact sequence
$$\cdots\longrightarrow M_2^{'}\longrightarrow M_1^{'}\longrightarrow M_0^{'}\longrightarrow L\longrightarrow 0\ \ (2)$$
where every $M_i^{'}\in{\mathcal{ET}}$.
Assembling the sequences $(1)$ and $(2)$, we
get the exact sequence
$$\cdots\rightarrow M_2^{'}\rightarrow M_1^{'}\rightarrow M_0^{'}\rightarrow N\rightarrow G\rightarrow 0,$$ where $N, M_i^{'}\in{\mathcal{ET}}$,
 as desired.}
\end{proof}
\begin{cor}\label{2.ty}
Let $R$ be an $(n,T)$-coherent ring and $G\in{\rm Gen}T$ a module. Then the
following assertions are equivalent:
\begin{enumerate}
\item [\rm (1)]
$G$ is Gorenstein $T_n^{d}$-flat;
\item [\rm (2)]
There is an exact sequence $0\rightarrow G\rightarrow N^{0}\rightarrow N^{1}\rightarrow \cdots$ of $R$-modules, where every $N^i\in{\mathcal{FT}}$;
\item [\rm (3)]
There is a short exact sequence $0\rightarrow G\rightarrow M\rightarrow K\rightarrow 0$ of $R$-modules, where
$M\in{\mathcal{FT}}$ and $K $ is Gorenstein $T_n^{d}$-flat.
\end{enumerate}
\end{cor}
\begin{proof}
{
$(1)\Longrightarrow (2)$ and $(1)\Longrightarrow (3)$ follow from definition.

$(2)\Longrightarrow (1)$ For any $R$-module $G\in{\rm Gen}T$, there is an exact sequence
$$\cdots\longrightarrow T_{1}\longrightarrow T_{0}\longrightarrow G\longrightarrow 0, $$
where any $T_{i}\in{\rm Add}T\subseteq {\mathcal{FT}}$. Thus, the exact sequence
$$ \cdots\longrightarrow T_1\longrightarrow T_{0}\longrightarrow N^0\longrightarrow
N^1\longrightarrow\cdots$$
of  ${\mathcal{FT}}$-modules  exists, where $G={\rm ker}(N^0\rightarrow N^1)$. Therefore by Theorem \ref{2.3}, $G$ is Gorenstein $T_n^{d}$-flat, 

$(3)\Longrightarrow (2)$ Assume that the exact sequence
$$ 0\longrightarrow G\longrightarrow M \longrightarrow K\longrightarrow 0 \ \ (1)$$ exists, where
$M\in{\mathcal{FT}}$ and $K $ is Gorenstein $T_n^{d}$-flat. Since $K$ is Gorenstein $T_n^{d}$-flat, there is an exact sequence
$$0\longrightarrow K\longrightarrow (N^0)^{'}\longrightarrow (N^1)^{'}\longrightarrow (N^2)^{'}\longrightarrow \cdots\ \ (2)$$
where every $(N^i)^{'}\in{\mathcal{FT}}$.
Assembling the sequences $(1)$ and $(2)$, we
get the exact sequence
$$0\rightarrow G\rightarrow M\rightarrow (N^0)^{'}\rightarrow (N^1)^{'}\rightarrow (N^2)^{'}\rightarrow \cdots,$$ where $M, (N^i)^{'}\in{\mathcal{FT}}$, as desired.
}
\end{proof}

\begin{prop}\label{2.5}
Let $G$ be a module. Then:
\item [\rm (1)]
If $G\in{\rm Cogen}T$ is Gorenstein $T_n^{d}$-injective, then ${\mathcal{E}}_{T}^{i}(U, G)=0$
for any $i>d$ and every $U\in {\rm F.Pres}^nT$  with ${\rm T.pdim}(U) <\infty$.
\item [\rm (2)]
If $0\rightarrow G\rightarrow G_{0}\rightarrow G_{1}\rightarrow\cdots \rightarrow G_{m-1}\rightarrow N\rightarrow 0$ is an exact sequence of
modules where every $G_j$ is a Gorenstein $T_n^{d}$-injective and $G_j\in{\rm Cogen}T$, then  ${\mathcal{E}}_{T}^{i}(U, N)={\mathcal{E}}_{T}^{m+i}(U, G)$ for any $i>d$ with ${\rm T.pdim} (U) <\infty$.
\item [\rm (3)]
If $G\in{\rm Gen}T$ is Gorenstein $T_n^{d}$-flat, then ${\Gamma}_{i}^{T}(U, G)=0$
for any $i>d$ and every $U\in {\rm F.Pres}^nT$ with ${\rm T.fdim}(U) <\infty$.
\item [\rm (4)]
If $0\rightarrow N\rightarrow G_{m-1}\rightarrow G_{m-2}\rightarrow\cdots \rightarrow G_{0}\rightarrow G\rightarrow 0$ is an exact sequence of
modules where every $G_i$ is a Gorenstein $T_n^{d}$-flat and $G_i\in{\rm Gen}T$, then ${\Gamma }_{i}^{T}(U, N)={\Gamma}_{m+i}^{T}(U, G)$  with ${\rm T.fdim} (U) <\infty$.
\end{prop}
\begin{proof}
{
(1)  Let $G$ be a Gorenstein $T_n^{d}$-injective $R$-module, and ${\rm T.pdim}(U)=m<\infty$.
Then by hypothesis, the following ${\mathcal{ET}}$-resolution of $G$ exists:
$$0\rightarrow L\rightarrow M_{m-1}\rightarrow \cdots \rightarrow M_{0}\rightarrow G\rightarrow 0.$$
So, ${\rm {\mathcal{E}}}_{T}^{i}(U, M_j)= 0$ for every $0 \leq j \leq m-1$ and any $i>d$, since $U\in {\rm F.Pres}^nT$ and any $M_j\in{\mathcal{ET}}.$  Thus by \cite[Proposition 2.2]{shaveisicam}, we deduce that
 ${\rm {\mathcal{E}}}_{T}^{i}(U,G)\cong{\rm {\mathcal{E}}}_{T}^{m+i}(U, L)$. Therefore ${\mathcal{E}}_{T}^{i}(U, G)=0$, since ${\rm T.pdim}(U)=m<\infty$.

(2) Setting $G_m=N$ and $K_{j-1}=\ker (G_{j-1}\rightarrow G_{j})$, for every $0\leq j\leq m$, the short exact sequence
$0\rightarrow K_{j-1}\rightarrow G_{j-1}\rightarrow K_{j}\rightarrow 0$ exists. Thus by (1), the induced exact sequences
$$0={\rm {\mathcal{E}}}_R^{i}(U,G_{j-1})\rightarrow{\rm {\mathcal{E}}}_R^{i}(U,K_{j})\rightarrow{\rm {\mathcal{E}}}_R^{i+1}(U,K_{j-1})\rightarrow
{\rm {\mathcal{E}}}_R^{i+1}(U,G_{j-1})=0$$
exists and so ${\rm {\mathcal{E}}}_T^{i}(U,K_{j})\cong{\rm {\mathcal{E}}}_R^{i+1}(U,K_{j-1})$. Since $K_0=G$, we have
$${\rm {\mathcal{E}}}_R^{m+i}(U,N)\cong{\rm {\mathcal{E}}}_R^{m+i-1}(U,K_{m-1})\cong\cdots\cong {\rm {\mathcal{E}}}_R^{i}(U,G),$$
as desired.

(3) and (4) are similar to the proof of (1) and (2).
}
\end{proof}


\begin{lem}\label{2.57}
Let $ 0\rightarrow A\stackrel{\displaystyle f}\rightarrow B\stackrel{\displaystyle g}\rightarrow C\rightarrow 0$ be an exact sequence. Then
\begin{enumerate}
\item [\rm (1)]
If $A$ is $T$-injective and $A, B, C \in {\rm Cogen}T$, then $B=A\oplus C$.
\item [\rm (2)]
If $A\in{\rm Copres}^{n}T$ and $C\in{\rm Copres}^{n}T$, then $B\in{\rm Copres}^{n}T.$
\item [\rm (3)]
If $C\in{\rm Copres}^{n}T$ and $B\in{\rm Copres}^{n+1}T$, then $A\in{\rm Copres}^{n+1}T.$
\item [\rm (4)]
If $B\in{\rm Copres}^{n}T$ and $A\in{\rm Copres}^{n+1}T$, then $C\in{\rm Copres}^{n}T.$
\end{enumerate}
\end{lem}
\begin{proof}
{
(1) 
If $A$ is $T$-injective and $A, B, C \in {\rm Cogen}T$, then we deduce that the
sequence
$$ 0\longrightarrow {\rm Hom}(C,A)\stackrel{\displaystyle g^{*}}\longrightarrow {\rm Hom}(B,A)\stackrel{\displaystyle f^{*}}\longrightarrow
{\rm Hom}(A,A)\longrightarrow {\mathcal{E}}_{T}^{1}(C,A)=0$$
is exact. So, 
there exists $h: B\rightarrow A$ such that $hf=1_{A}.$

(2)
It is similar to the proof of Lemma \ref {2.io}, Part (1).

(3) Let $B\in{\rm Copres}^{n+1}T$ and $C\in{\rm Copres}^nT$,  then the following commutative diagram with exact rows:
\begin{center}
$
 \begin{array}{ccccccccc}
&{\hspace{-7.4cm}0}&{\hspace{-6.cm}0}&  \\
    &{\hspace{-7.4cm}\downarrow} &{\hspace{-6cm}\downarrow}&  \\
  \hspace{-3.3cm}0\longrightarrow A&\hspace{-5.9cm}= {\hspace{-1.8mm}=}  A&&&  \\
    \hspace{-2.3cm}\downarrow & \hspace{-5.3cm}\downarrow &  \\
    0\longrightarrow B \longrightarrow T_{0}\longrightarrow L\longrightarrow 0 \\
    \hspace{-2.3cm}\downarrow &  &\hspace{-6cm}\downarrow &\hspace{-4.5cm}\parallel  &  \\
      0\longrightarrow C\longrightarrow D\longrightarrow L\longrightarrow 0& \\
 &&\hspace{-8.cm}\downarrow &\hspace{-6.5cm}\downarrow&  \\
  & &  \hspace{-8.cm}  0&\hspace{-6.5cm}  0
     \end{array}
$
\end{center}
\noindent exists, where $T_{0} \in {\rm Prod}T$ and $L\in {\rm Copres}^{n}T$. By (2), $D\in {\rm Copres}^{n}T$. So, we deduce that $A\in {\rm Copres}^{n+1}T$.

(4)  Let $A\in{\rm Copres}^{n+1}T$ and $B\in{\rm Copres}^nT$,  then the following commutative diagram with exact rows:
\begin{center}
$
 \begin{array}{ccccccccc}
&{\hspace{-7.8cm}0}&{\hspace{-6.4cm}0}&  \\
    &{\hspace{-7.7cm}\downarrow} &{\hspace{-6.4cm}\downarrow}&  \\
  0\longrightarrow A\longrightarrow T_{0}^{'}\longrightarrow L^{'}\longrightarrow 0   \\
    \hspace{-2.4cm}\downarrow & \hspace{-5.5cm}\downarrow &  \\
    0\longrightarrow B \longrightarrow T_{0}\longrightarrow L\longrightarrow 0 \\
    \hspace{-2.3cm}\downarrow &  &\hspace{-6.cm}\downarrow &\hspace{-4.5cm}\parallel    &  \\
      0\longrightarrow C\longrightarrow D\longrightarrow L\longrightarrow 0& \\
 &&\hspace{-8.cm}\downarrow &\hspace{-6.6cm}\downarrow&  \\
  & &  \hspace{-8.cm}  0&\hspace{-6.6cm}  0
     \end{array}
$
\end{center}
\noindent exists, where $T_{0}, T_{0}^{'}  \in {\rm Prod}T$ and $L\in {\rm Copres}^{n-1}T$. Since $T_{0}^{'}$ is $T$-injective, we have that $T_{0}=T_{0}^{'}\oplus D$ by (1), and $D\in{\rm Cogen}T$. Thus for any module $N$, we have
$$ {\mathcal{E}}_{T}^{1}(N,T_{0})={\mathcal{E}}_{T}^{1}(N,T_{0}^{'}\oplus D)={\mathcal{E}}_{T}^{1}(N,T_{0}^{'})\oplus {\mathcal{E}}_{T}^{1}( N,D)=0.$$ 
Hence $D\in {\rm Prod}T$. On the other hand, $L\in {\rm Copres}^{n-1}T$. Therefore, we conclude that $C\in {\rm Copres}^{n}T.$
}
\end{proof}
\begin{prop}\label{2.7}
Let $R$ be $(n,T)$-coherent.
\begin{enumerate}
\item [\rm (1)]
Let $0\rightarrow M\rightarrow G\rightarrow N\rightarrow 0$ be an exact sequence. If $N\in{\rm Cogen}T$ is Gorenstein $T_n^d$-injective and $M\in{\mathcal{ET}}$, then $G$ is Gorenstein $T_n^d$-injective.
\item [\rm (2)]
Let $0\rightarrow K\rightarrow G\rightarrow N\rightarrow 0$ be an exact sequence. If $K\in{\rm Gen}T$ is Gorenstein $T_n^d$-flat and $N\in{\mathcal{FT}}$, then $G$ is Gorenstein $T_n^d$-flat.

\end{enumerate}
\end{prop}
\begin{proof}
{

(1) By Lemma \ref{2.57}, $G\in{\rm Cogen}T$, since $M, N \in{\rm Cogen}T$.  $N$ is Gorenstein $T_n^d$-injective. So by Corollary $\ref{2.h}$, there exists an exact sequence of
$0\rightarrow K\rightarrow M^{'}\rightarrow N\rightarrow 0$, where $M^{'}\in{\mathcal{ET}}$ and $K$ is Gorenstein $T_n^d$-injective. 
Now, we consider the following diagram:
\begin{center}
$
\begin{array}{ccccccccc}
&& & & 0 & & 0 & & \\
&& & & \downarrow & & \downarrow & & \\
&& & & K & ={\hspace{-1.5mm}=} & K & & \\
&& & & \downarrow & & \downarrow & & \\
0 &\longrightarrow & M & \longrightarrow & D &\longrightarrow &M' & \longrightarrow & 0 \\
& & \parallel& & \downarrow & & \downarrow & & \\
0 & \longrightarrow & M& \longrightarrow & G& \longrightarrow & N & \longrightarrow & 0 \\
& & & & \downarrow & & \downarrow & & \\
&& & & 0 & & 0 & & \\
\end{array}
$
\end{center}
The exactness of the middle horizontal sequence with $M, M^{'}\in{\mathcal{ET}}$, implies that $D\in{\mathcal{ET}}$. Hence from the middle vertical sequence
and Corollary $\ref{2.h}$, we deduce that $G$ is Gorenstein $T_n^d$-injective.

(2)
By Lemma \ref{2.io}, $G\in{\rm Gen}T$, since $K, N \in{\rm Gen}T$.  $K$ is Gorenstein $T_n^d$-flat. So  by Corollary $\ref{2.ty}$, there exists an exact sequence of
$0\rightarrow K\rightarrow N^{'}\rightarrow L\rightarrow 0$, where $N^{'}\in{\mathcal{FT}}$ and $L$ is Gorenstein $T_n^d$-flat.
Now, we consider the following diagram:
\begin{center}
$
\begin{array}{ccccccccc}
&&  0 & & 0 & & \\
&&  \downarrow & & \downarrow & & \\
& 0 \longrightarrow & K &\longrightarrow &G & \longrightarrow &N&\longrightarrow &  0 \\
& & \downarrow & & \downarrow& & \parallel & & \\
& 0\longrightarrow &N'& \longrightarrow & E & \longrightarrow&N & \longrightarrow &0 \\
& & \downarrow & & \downarrow & & \\
&&  L & ={\hspace{-1.5mm}=} & L & & \\
& &  \downarrow & & \downarrow & & \\
&&  0 & & 0 & & \\
\end{array}
$
\end{center}
The exactness of the middle horizontal sequence with $N, N^{'}\in{\mathcal{FT}}$, implies that $E\in{\mathcal{FT}}$. Hence from the middle vertical sequence
and Corollary $\ref{2.ty}$, we deduce that $G$ is Gorenstein $T_n^d$-flat.

}
\end{proof}

In this part, we show that which conditions under every module in ${\rm Cogen}T$  is
Gorenstein $T_n^d$-injective.
\begin{prop}\label{3.1}
Let $R$ be a ring. The following assertions are equivalent:
\begin{enumerate}
\item [\rm (1)]
Every module in ${\rm Cogen}T$, is Gorenstein $T_n^d$-injective;
\item [\rm (2)]
The ring satisfies the following two conditions:

{\rm (i)} Every $T$-projective module is $T_n^d$-injective.

{\rm (ii)} ${\mathcal{E}}_T^{d+1}(U, N)=0$ for any $N\in{\rm Cogen}T$ and any $U\in{\rm F.Pres}^nT$ with ${\rm T.pdim}(U) <\infty$.
\end{enumerate}
\end{prop}
\begin{proof}
{
$(1)\Longrightarrow (2)$ The condition $(i)$ follows from this fact that every $T$-projective
module $M$ is Gorenstein $T_n^d$-injective. So, the following ${\mathcal{ET}}$-resolution of $M$ exists:
$$\cdots \rightarrow M_1\rightarrow M_0\rightarrow M\rightarrow 0.$$
Since $M$ is $T$-projective, $M$ is
$T_n^d$-injective as a direct summand of $M_0$. Also,  by Proposition \ref{2.5} and (1),  the condition $(ii)$ follows.

$(2)\Longrightarrow (1)$ Since $T$ is tilting, the exact sequence 
$0\rightarrow R\rightarrow T_{0}\rightarrow T_{1}\rightarrow 0$ exists, where $T_{0}, T_{1}\in{\rm Add}T$. So $T_{0}, T_{1}\in{\rm Gen}T$. Hence, the exact sequence $0\rightarrow K\rightarrow T^{(m)}\rightarrow T_{i}\rightarrow 0$ exists for $i=0,1$. On the other hand, 
$K\subseteq T^{(m)}\subseteq T^{m}$. So, $K,T^{(m)}\in{\rm Cogen}T $. Thus by \cite[Proposition 2.1]{shaveisicam}, $K,T^{(m)}\in{\rm Copres}^\infty T$. Therefore by Lemma \ref{2.57}, $T_{i},R\in{\rm Copres}^kT$ and hence $R\in{\rm Cogen}T$. 
 Let $G\in{\rm Cogen}T$. Choose a ${\rm Prod}T$-resolution $0\rightarrow G\rightarrow T^{0}\rightarrow T^{1}\rightarrow\cdots$ of $G$ and a free resolution $\cdots \rightarrow F_{1}\rightarrow F_{0}\rightarrow G\rightarrow 0$, where every $F_i\in{\rm Cogen}T$. Also by 
Lemma \ref{2.io} and \cite[Proposition 2.1]{shaveisicam}, we get that $F_i\in{\rm Gen}T={\rm Pres}^\infty T$, since $T_{0}, T_{1}\in{\rm Gen}T$. Every projective in ${\rm Gen}T$ is $T$-projective. So by (2), every $F_i$ is $T_{n}^d$-injective. 
 Assembling these resolutions, by Remark \ref{2} and (2)(i), we get the following ${\mathcal{ET}}$-resolotion:
$${\mathbf{A}}=\cdots \rightarrow F_{1}\rightarrow F_{0}\rightarrow T^{0}\rightarrow T^{1}\rightarrow\cdots,$$
where $G={\rm ker}(T^0\rightarrow T^1)$, $K^{i}={\rm ker}(T^i\rightarrow T^{i+1})$ and $K_{i}={\rm ker}(F_{i}\rightarrow F_{i-1})$ for any $i\geq 1$. By  Lemma \ref{2.57}, $ K_{i}, K^{i}\in{\rm Cogen}T $, since $G, T^{i}, F_{i}\in{\rm Cogen}T $. Let $U\in{\rm F.Pres}^nT$ with ${\rm T.p.dim}(U) <\infty$. Then by (2), ${\mathcal{E}}_T^{d+1}(U, G)={\mathcal{E}}_T^{d+1}(U, F_i)={\mathcal{E}}_T^{d+1}(U, T^{i})=0$ for any $i\geq 0$.
So, ${\mathcal{E}}_T^{d}(U,{\mathbf{A}})$ is exact, and hence $G$ is Gorenstein $T_n^d$-injective.}
\end{proof}

\begin{thm}\label{3.2}
Let $R$ be an $(n,T)$-coherent ring. Then the following are equivalent:
\begin{enumerate}
\item [\rm (1)]
Every module in ${\rm Cogen}T$, is Gorenstein $T_n^d$-injective;
\item [\rm (2)]
Every $T$-projective module is $T_n^d$-injective;
\item [\rm (3)]
$R$ is $T_n^d$-injective;
\item [\rm (4)]
Every Gorenstein $T_n^d$-flat is Gorenstein $T_n^d$-injective;
\item [\rm (5)]
Every $T$-flat module is Gorenstein $T_n^d$-injective;
\item [\rm (6)]
Every $T$-projective module is Gorenstein $T_n^d$-injective.
\end{enumerate}
\end{thm}
\begin{proof}
{$(1)\Longrightarrow (2)$ and $(2)\Longrightarrow (3)$, is hold by Proposition \ref{3.1}. 

$(3)\Longrightarrow (1)$
Let $G\in{\rm Cogen}T$ be a module and $\cdots\rightarrow F_1\rightarrow F_0\rightarrow G\rightarrow 0$ be any
free resolution of $G$. Then, similar to proof ($(2)\Longrightarrow (1)$) of Proposition \ref{3.1},   each $F_i\in{\mathcal{ET}}$. Hence Corollary $\ref{2.h}$ completes the proof.

$(3)\Longrightarrow (4)$ Let $N\in{\rm Gen}T$ is Gorenstein $T_n^d$-flat. Similar to proof ($(2)\Longrightarrow (1)$) from Proposition \ref{3.1}, $N\in{\rm Cogen}T$. So, (2)  follows immediately from (1).

$(4)\Longrightarrow (5)$
every $T$-flat is $T_n^d$-flat and every $T_n^d$-flat is Gorenstein $T_n^d$-flat. So by (4), (5) is hold.

$(5)\Longrightarrow (6)$ Is clear, since every $T$-projective is $T$-flat.

$(6)\Longrightarrow (3)$ Similar to proof ($(1)\Longrightarrow (2)$) from Proposition \ref{3.1}, every $T$-projective module is $T_n^d$-injective. Also, the exact sequence 
$0\rightarrow R\rightarrow T_{0}\rightarrow T_{1}\rightarrow 0$ exists, where $T_{0}, T_{1}\in{\rm Add}T$. So $T_{0}, T_{1}\in{\rm Gen}T$. Thus by \cite[Proposition 2.1]{shaveisicam},  $T_i\in{\rm Gen}T={\rm Pres}^\infty T$ for $i=0,1$. Hence by Lemma \ref{2.io}, $R\in{\rm Gen}T$. Therefore $R$ is $T$-projective and hence, it is $T_n^d$-injective.

}
\end{proof}

Let $R$ be a ring. Then $R$ is $n$-regular if every $n$-presented $R$-module is projective. $M$ is called  $(n,d)$-flat  if ${\rm Tor}_{d+1}^{R}(U,M)=0$ for every $n$-presented $R$-module $U$.
 $M$ is called  $(n,d)$-injective if ${\rm Ext}_{R}^{d+1}(U,M)=0$ for every $n$-presented $R$-module $U$ (see, \cite{ N.I, X.J}).  In particular, if $T=R$, then every $T_n^{d}$-flat module is $(n,d)$-flat, every $T_n^{d}$-injective module is $(n,d)$-injective, every Gorenstein $T_n^{d}$-flat module is  Gorenstein $(n,d)$-flat and every  Gorenstein $T_n^{d}$-injective module Gorenstein $(n,d)$-injective.
\begin{example}\label{3.54} 
\begin{enumerate}
\item [\rm (1)]
{Let $R$ be a $1$-Gorenstein ring and
$0\rightarrow R\rightarrow E^{0}\rightarrow  E^{1} \rightarrow 0$
be  the minimal injective resolution of $R$. Then, $T=E_{0}\oplus E_{1}$ is Gorenstein $T_n^d$-injective and Gorenstein $T_n^d$-flat, since  by \cite{E.E}, $T$ is a tilting
module. 
\item [\rm (2)]
Let $R$ be an $n$-regular ring. Then replacing $T$ by $R$ as an $R$-module of Theorem \ref{3.2},   every $R$-module is Gorenstein $(n,0)$-injective and Gorenstein $(n,0)$-flat, since by \cite[Theorem 3.9]{N.I}, $R$ is $(n,0)$-injective.}
\end{enumerate}
\end{example}

{}

\end{document}